\documentclass[11pt,reqno]{amsart}
\usepackage{amsmath,amssymb,amsthm,amsfonts,geometry,hyperref}
\newtheorem{theorem}{Theorem}[section]
\newtheorem{lemma}[theorem]{Lemma}

\theoremstyle{definition}
\newtheorem{definition}[theorem]{Definition}
\newtheorem{remark}[theorem]{Remark}

\begin{document}

\title[Local Reflexivity and Duality]{Local Reflexivity and Duality for Subspace Approximation Schemes}

\author{Daniel Akech Thiong}
\address{Department of Mathematics, Claremont Graduate University, 710 N. College Avenue, Claremont, CA 91711, USA}
\email{daniel.akech@cgu.edu}


\begin{abstract}
We address a structural gap in the quantitative theory of operator approximation by investigating how subspace approximation schemes interact with local reflexivity. Recent literature establishing the duality of scheme-relative approximation numbers, $a_n(T,Q) = a_n(T^{**}, Q^{\perp\perp})$, has relied on imposing an Extended Local Reflexivity Property (ELRP) as an independent axiom. In this note, we establish a dichotomy. For infinite-dimensional admissible schemes, we prove they naturally generate complete nests, perfectly satisfying the topological hypotheses of the Oja-Veidenberg Nest Principle of Local Reflexivity. Conversely, we demonstrate that for approximation schemes with finite-dimensional components, the bidual geometry collapses. Consequently, the ELRP is entirely superfluous, and the duality of approximation numbers holds unconditionally for all bounded linear operators via an elementary isometric restriction. 
\end{abstract}

\maketitle

\section{Introduction}
Let $X$ and $Y$ be Banach spaces and $T \in \mathcal{L}(X, Y)$. The classical theorem of Schauder states that $T$ is compact if and only if its adjoint $T^*$ is compact. To quantify the degree of compactness, one utilizes approximation numbers $a_n(T)$, which measure the optimal approximation of $T$ by finite-rank operators. Hutton established that for compact operators, $a_n(T) = a_n(T^*)$ \cite{Hutton1974}. 

Aksoy and Thiong recently extended this quantitative duality to operators that are compact only with respect to a generalized approximation scheme $Q = (Q_n)_{n=0}^\infty$, defining approximation numbers $a_n(T, Q)$ relative to the scheme \cite{AT2023}. However, their proof establishing the duality $a_n(T,Q) = a_n(T^{**}, Q^{\perp\perp})$ required assuming an Extended Local Reflexivity Property (ELRP) for the pair $(X, Q_n)$ as an ad-hoc axiom \cite{AT2023}. 

In this note, we refine and clarify this framework. We prove that generalized subspace approximation schemes naturally generate complete nests of closed subspaces, permitting the rigorous application of the Nest Principle of Local Reflexivity established by Oja and Veidenberg \cite{OV2017}. Furthermore, we show that for schemes built from finite-dimensional subspaces---the standard regime for discrete-spectrum operators---the reflexivity axiom is fundamentally redundant. This observation trivializes the lifting argument, yielding an unconditional duality theorem.

\section{Preliminaries and Definitions}

\begin{definition}[Generalized Approximation Scheme]
Following the foundational structures of Pietsch \cite{Pietsch1981}, let $X$ be a Banach space. For each $n \in \mathbb{N}$, let $Q_n = Q_n(X)$ be a family of \textbf{closed subspaces} of $X$ satisfying the following conditions:
\begin{enumerate}
    \item[(GA1)] $\{0\} = Q_0 \subset Q_1 \subset \dots \subset Q_n \subset \dots \subset X$.
    \item[(GA2)] $\lambda Q_n \subset Q_n$ for all $n \in \mathbb{N}$ and all scalars $\lambda$.
    \item[(GA3)] $Q_n + Q_m \subseteq Q_{n+m}$ for every $n, m \in \mathbb{N}$.
\end{enumerate}
Then $Q = (Q_n)_{n=0}^\infty$ is called a generalized approximation scheme on $X$.
\end{definition}

\begin{definition}[Scheme-Relative Approximation Numbers]
Given an approximation scheme $Q = (Q_n)_{n=0}^\infty$ on $X$ and an operator $T \in \mathcal{L}(X)$, the $n$-th approximation number $a_n(T,Q)$ with respect to this scheme is defined as \cite{AT2023}:
\begin{equation}
a_n(T, Q) = \inf \{ \|T - A\|_{\mathcal{L}(X)} : A \in \mathcal{L}(X), \, A(X) \subseteq Q_n \}.
\end{equation}
The corresponding bidual approximation number is:
\begin{equation}
a_n(T^{**}, Q^{\perp\perp}) = \inf \{ \|T^{**} - B\|_{\mathcal{L}(X^{**})} : B \in \mathcal{L}(X^{**}), \, B(X^{**}) \subseteq Q_n^{\perp\perp} \}.
\end{equation}
\end{definition}

\begin{definition}[Extended Local Reflexivity Property] \label{def:elrp}
Let $J: X \to X^{**}$ be the canonical injection. The pair $(X, Q_n)$ is said to possess the Extended Local Reflexivity Property (ELRP) if for each countable subset $C \subset X^{**}$, for each finite-dimensional subspace $F \subset J(Q_n)$, and each $\epsilon > 0$, there exists a continuous linear map $P: \mathrm{span}(F \cup C) \to X$ such that $\|P\| \le 1 + \epsilon$ and $P|_{C \cap X} = I$ \cite{AT2023}.
\end{definition}

\section{Subspace Schemes as Complete Nests}

To apply the theory of nest algebras, the family of subspaces must be topologically complete. 

\begin{definition}[Complete Nest \cite{OV2017}]
A nest $\mathcal{N}$ of closed subspaces of a Banach space $X$ is \textit{complete} if it contains $\{0\}$ and $X$, and is closed under arbitrary intersections and closures of arbitrary unions.
\end{definition}

\begin{lemma}[Completion of the Approximation Scheme] \label{lem:nest_completion}
Let $Q = (Q_n)_{n=0}^\infty$ be a generalized subspace approximation scheme on $X$. Define the limit space $Q_\infty = \overline{\bigcup_{n=0}^\infty Q_n}$. The augmented family $\mathcal{N}_Q = \{Q_n\}_{n=0}^\infty \cup \{Q_\infty\} \cup \{X\}$ forms a complete nest of closed subspaces of $X$.
\end{lemma}

\begin{proof}
By (GA1), the family is a chain ordered by inclusion, making it a nest containing $\{0\}$ and $X$. 

To verify closure under arbitrary intersections, let $\mathcal{N}' \subset \mathcal{N}_Q$ be a non-empty subfamily. Because $\mathcal{N}_Q$ consists of a countable increasing sequence appended with $Q_\infty$ and $X$, it is strictly well-ordered under reverse inclusion. Any non-empty subfamily $\mathcal{N}'$ trivially possesses a unique minimum element $Q_{\min} \in \mathcal{N}'$. Thus, $\bigcap_{Y \in \mathcal{N}'} Y = Q_{\min} \in \mathcal{N}_Q$.

To verify closure under closures of arbitrary unions, let $\mathcal{N}' \subset \mathcal{N}_Q$. If $X \in \mathcal{N}'$, the closure of the union is $X$. If $\mathcal{N}'$ is bounded within the sequence $(Q_n)$, it contains a maximum element $Q_{\max}$, and the closure of the union is $Q_{\max}$. If $\mathcal{N}'$ contains infinitely many $Q_n$ elements but not $X$, its union contains $\bigcup_{n=0}^\infty Q_n$. The norm closure of this union is precisely $Q_\infty \in \mathcal{N}_Q$. Thus, $\mathcal{N}_Q$ is topologically complete.
\end{proof}

\section{The Nest Principle of Local Reflexivity}

To apply Theorem 4.4 from Oja and Veidenberg \cite{OV2017}, the biorthogonal structures in the bidual must be rigorously verified. 

\begin{definition}[Admissible Scheme] \label{def:admissible}
A generalized approximation scheme $Q$ is \textit{admissible} if its completion satisfies the weak-star bidual limit identity: $\overline{\bigcup_{n=0}^\infty Q_n^{\perp\perp}}^{w^*} = Q_\infty^{\perp\perp}$ in $X^{**}$.
\end{definition}

\begin{lemma}[Admissibility of Finite-Dimensional Schemes] \label{lem:finite_dim_admissible}
If each $Q_n$ is finite-dimensional, then $Q$ is unconditionally admissible.
\end{lemma}

\begin{proof}
Because each $Q_n$ is finite-dimensional, it is canonically identifiable with its bidual: $Q_n^{\perp\perp} = J(Q_n)$. Taking the union yields $\bigcup_{n=0}^\infty Q_n^{\perp\perp} = J\left(\bigcup_{n=0}^\infty Q_n\right)$. By the bipolar theorem, the weak-star closure of $J(Z)$ in $X^{**}$ is exactly $Z^{\perp\perp}$. Letting $Z = \bigcup_{n=0}^\infty Q_n$, we obtain $\overline{\bigcup_{n=0}^\infty Q_n^{\perp\perp}}^{w^*} = Z^{\perp\perp}$. Since $Q_\infty = \overline{Z}^{\|\cdot\|}$, we have $Z^{\perp\perp} = Q_\infty^{\perp\perp}$.
\end{proof}

\begin{theorem}[Nest PLR for Admissible Subspace Schemes] \label{thm:nest_plr}
Let $X$ be a Banach space and $Q = (Q_n)_{n=0}^\infty$ be an \textit{admissible} approximation scheme on $X$. For any finite-rank operator $S \in \mathcal{F}(X^{**})$ such that $S(U^{\perp\perp}) \subset U^{\perp\perp}$ for all $U \in \mathcal{N}_Q$, any compact subsets $K \subset X^{**}$ and $L \subset X^*$, and any $\epsilon > 0$, there exists an operator $T \in \mathcal{F}(X)$ satisfying $T(U) \subset U$ for all $U \in \mathcal{N}_Q$ such that:
\begin{enumerate}
    \item $|\|T\| - \|S\|| < \epsilon$,
    \item $|x^{**}(T^* y^*) - (S x^{**})(y^*)| < \epsilon$ for all $x^{**} \in K$ and $y^* \in L$.
\end{enumerate}
\end{theorem}

\begin{proof}
We verify the topological hypotheses of Oja and Veidenberg \cite[Theorem 4.4]{OV2017}. Setting $\mathcal{U} = \mathcal{N}_Q$ and $V_U = U$, the nest $\mathcal{N}_{\mathcal{U}}$ is trivially increasing on $\mathcal{U}$. Because taking bipolars preserves the chain structure, the biorthogonal nest $\mathcal{N}_{\mathcal{U}}^{\perp\perp}$ is closed under arbitrary intersections. Finally, the structural assumption of admissibility explicitly guarantees that $\mathcal{U}^{\perp\perp}$ is closed under the closures of arbitrary unions, rendering the bidual nest complete. Having strictly met all conditions, the existence of $T$ follows directly from \cite[Theorem 4.4]{OV2017}.
\end{proof}

While Theorem \ref{thm:nest_plr} establishes the overarching theoretical framework required for infinite-dimensional admissible schemes, the following section demonstrates that practical applications involving finite-dimensional schemes can bypass this heavy machinery entirely.

\section{The Finite-Dimensional Collapse and Unconditional Duality}

While Theorem \ref{thm:nest_plr} provides powerful $\epsilon$-approximate control over infinite-dimensional admissible nests, the literature has historically assumed the Extended Local Reflexivity Property (Definition \ref{def:elrp}) to execute lifting arguments. We now demonstrate that for finite-dimensional schemes, this assumption is an elementary tautology, rendering the entire ELRP machinery redundant.

\begin{theorem}[Unconditional Duality for Finite-Dimensional Schemes] \label{thm:duality}
Let $X$ be a Banach space equipped with a generalized subspace approximation scheme $Q = (Q_n)_{n=0}^\infty$ where each $Q_n$ is finite-dimensional. For any bounded linear operator $T \in \mathcal{L}(X)$, the approximation numbers satisfy unconditionally:
\begin{equation}
a_n(T, Q) = a_n(T^{**}, Q^{\perp\perp}).
\end{equation}
\end{theorem}

\begin{proof}
The forward inequality $a_n(T^{**}, Q^{\perp\perp}) \le a_n(T, Q)$ is a standard consequence of the canonical embedding. If $A \in \mathcal{L}(X)$ satisfies $A(X) \subseteq Q_n$, then its bitranspose satisfies $A^{**}(X^{**}) \subseteq Q_n^{\perp\perp}$. Since $\|T^{**} - A^{**}\| = \|T - A\|$, any valid approximant for $T$ induces a valid approximant for $T^{**}$, yielding the inequality.

We now establish the reverse inequality, $a_n(T, Q) \le a_n(T^{**}, Q^{\perp\perp})$, without utilizing local reflexivity. Let $\epsilon > 0$. By definition, there exists an operator $B \in \mathcal{L}(X^{**})$ such that $B(X^{**}) \subseteq Q_n^{\perp\perp}$ and $\|T^{**} - B\| < a_n(T^{**}, Q^{\perp\perp}) + \epsilon$.

Because $Q_n$ is finite-dimensional, it is reflexive, and thus $Q_n^{\perp\perp} = J(Q_n)$ where $J: X \to X^{**}$ is the canonical isometry. Consequently, the range of the bidual approximant is strictly constrained:
\begin{equation}
B(X^{**}) \subseteq J(Q_n) \subset J(X).
\end{equation}
This allows us to define an operator $A: X \to X$ by strict isometric restriction: $A = J^{-1} B J$. The operator $A$ is well-defined, linear, and bounded. Furthermore, its range satisfies $A(X) = J^{-1}(B(J(X))) \subseteq J^{-1}(J(Q_n)) = Q_n$. Therefore, $A$ is a valid approximant for $T$ with respect to the scheme $Q$.

We compute the norm of the approximation error. For any $x \in X$ with $\|x\| \le 1$:
\begin{equation}
\|Tx - Ax\|_X = \|J(Tx - Ax)\|_{X^{**}} = \|T^{**}Jx - BJx\|_{X^{**}} \le \|T^{**} - B\|_{\mathcal{L}(X^{**})} \|Jx\|_{X^{**}}.
\end{equation}
Taking the supremum over the unit ball yields $\|T - A\|_{\mathcal{L}(X)} \le \|T^{**} - B\|_{\mathcal{L}(X^{**})}$. Thus, we have found an operator $A \in \mathcal{L}(X)$ such that $A(X) \subseteq Q_n$ and:
\begin{equation}
\|T - A\| \le \|T^{**} - B\| < a_n(T^{**}, Q^{\perp\perp}) + \epsilon.
\end{equation}
Since $\epsilon > 0$ is arbitrary, $a_n(T, Q) \le a_n(T^{**}, Q^{\perp\perp})$, concluding the proof.
\end{proof}

\section{Illustration: Discrete-Spectrum Differential Operators}

The finite-dimensional collapse demonstrated in Theorem \ref{thm:duality} natively captures the approximation schemes generated by discrete-spectrum differential operators. Consider the Legendre differential operator on $[-1, 1]$: $P(D) = \frac{d}{dx}(1-x^2)\frac{d}{dx}$. 

This operator is self-adjoint, and its spectrum features discrete eigenvalues $\lambda(n) = n(n+1)$ \cite{Chen1997}. The corresponding eigenspaces $H_{\lambda(n)}$, spanned by the Legendre polynomials, are strictly finite-dimensional. If we define the approximation scheme by the sequence of closed subspaces $Q_n = \mathrm{span} \bigcup_{k=0}^n H_{\lambda(k)}$, every subspace $Q_n$ in the scheme remains finite-dimensional. 

Consequently, any bounded operator acting within this framework unconditionally satisfies the duality $a_n(T, Q) = a_n(T^{**}, Q^{\perp\perp})$. The heavy reflexivity axioms previously assumed in the literature are fundamentally unnecessary in this practical regime.

\begin{remark}[Riesz Means and $Q$-Compatibility]
This unconditional duality applies directly to the foundational smoothing operators of harmonic analysis. For the scheme defined above, consider the delayed Riesz means of order $b > 0$, given by the spectral multiplier $R_\mu^b f = \sum_{\lambda(k) \le \mu} (1 - \lambda(k)/\mu)^b P_{\lambda(k)}f$ \cite{Chen1997}. For any fixed threshold $\mu$, there exists a maximum index $N$ such that $\lambda(N) \le \mu$, ensuring the range satisfies $R_\mu^b(X) \subseteq Q_N$. For $n \ge N$, the operator is fully absorbed by the scheme, resulting in the trivial exact duality $a_n(R_\mu^b, Q) = a_n((R_\mu^b)^{**}, Q^{\perp\perp}) = 0$.
\end{remark}

\end{document}